\documentclass{amsart}
\usepackage{amssymb,amsthm,amsmath}
\usepackage{graphicx}
\usepackage{color}
\usepackage[usenames,dvipsnames,svgnames,table]{xcolor}
\usepackage{tikz}
\usepackage{comment}
\usepackage{epsfig}
\newtheorem{lemma}{Lemma}[section]
\newtheorem{theorem}[lemma]{Theorem}

\newtheorem*{definition}{Definition}
\newtheorem*{remark}{Remark}

\renewcommand{\epsilon}{\varepsilon}
\begin{document}
\title[Orbit Portraits]{Orbit Portraits of Unicritical Anti-polynomials}

\author[S. Mukherjee]{Sabyasachi Mukherjee}

\address{Jacobs University Bremen, Campus Ring 1, Bremen 28759, Germany}

\email{s.mukherjee@jacobs-university.de}


\date{\today}

\begin{abstract}
Orbit portraits were introduced by Milnor as a combinatorial tool to describe the patterns of all periodic dynamical rays landing on a periodic cycle of a quadratic polynomial. This encodes information about the dynamics and the parameter spaces of these maps. We carry out a similar analysis for unicritical anti-polynomials and give an explicit description of the orbit portraits that can occur for such maps in terms of their characteristic angles, which turns out to be rather restricted when compared with the holomorphic case. Finally, we prove a realization theorem for these combinatorial objects. The results obtained in this paper serve as a combinatorial foundation for a detailed understanding of the combinatorics and topology of the parameter spaces of unicritical anti-polynomials and their connectedness loci, known as the multicorns.
\end{abstract}

\maketitle

\tableofcontents

\section{Introduction}

In this article, we study some combinatorial properties of the iteration of unicritical anti-polynomials $f_c(z) = \bar{z}^d + c,$ for any degree $d \geq 2 $ and $c \in \mathbb{C}$. As in polynomial iteration theory, the dynamics of an anti-polynomial on the plane induces a simpler dynamical system on $\mathbb{R}/\mathbb{Z}$  via its action on the dynamical rays. This is a combinatorial object that one can study with greater ease and recover many properties of the original dynamical system. One of the key tools in this approach is the notion of orbit portraits, which was first introduced by John Milnor in his paper \cite{M2a} to describe the pattern of all periodic dynamical rays landing at different points of a periodic cycle for quadratic complex polynomials. The usefulness of orbit portraits stems from the fact that these combinatorial objects contain substantial information on the connection between the dynamical and the parameter planes of the maps under consideration.

We define the set of all points which remain bounded under all iterations of an anti-polynomial $f_c$ to be the \emph{filled-in Julia set} $K(f_c)$. The boundary of the filled-in Julia set is defined to be the \emph{Julia set} $J(f_c)$ and the complement of the Julia set is defined to be its \emph{Fatou set} $F(f_c)$. This leads, as in the holomorphic case, to the notion of the \emph{Connectedness Locus} of degree $d$ unicritical anti-polynomials:

\begin{definition}
The\emph{ multicorn} of degree $d$ is defined as $\mathcal{M}^{\ast}_d = \lbrace c \in \mathbb{C} : K(f_c)$ is connected$\rbrace$.
\end{definition} 

The dynamics of anti-quadratic maps and their connectedness locus, the tricorn, was first studied in \cite{CHRS} and their numerical experiments showed differences between the Mandelbrot set and the tricorn in that there are bifurcations from the period $1$ hyperbolic component to the period $2$ hyperbolic components along arcs in the tricorn, but not in the Mandelbrot set. Milnor found small tricorn-like sets in the parameter space of real cubic polynomials \cite{M3a}. Later, Nakane and Schleicher, in \cite{NS}, studied the structure of hyperbolic components  of $\mathcal{M}_d^*$ via the multiplier map (even period case) and the critical value map (odd period case). These maps are branched coverings over the unit disk of degree $d-1$ and $d+1$ respectively, branched only over the origin. Many years down the line, Hubbard and Schleicher \cite{HS} proved that the multicorns are not path connected, confirming a conjecture of Milnor.


The combinatorics and topology of the multicorns differ in many ways from those of their holomorphic counterparts, the multibrot sets, which are the connectedness loci of degree $d$ unicritical polynomials \cite{NS, HS}. At the level of combinatorics, this manifests itself in the structure of orbit portraits (Theorem \ref{complete antiholomorphic}). These combinatorial results are extensively used in a forthcoming paper \cite{MNS}, where we study the bifurcation phenomena, the structure of boundaries of odd period hyperbolic components, the combinatorics of parameter rays, the discontinuity of landing points of dynamical rays, the number of hyperbolic components of a given period to name a few. Since we lose holomorphic dependence in the parameter spaces of anti-polynomials, one needs to carry out many proofs using more combinatorial methods. One of the key ingredients of the proofs in \cite{MNS} is the orbit portraits, and these are the main subject of the present manuscript. Another place where the combinatorial results proved in this paper are utilized is \cite{IM}, where we prove that certain parameter rays of the multicorns do not land; rather they accumulate on an arc of positive length in the parameter plane.

The following is an overview of the results proved in this paper. After giving some necessary background, we define orbit portraits for unicritical anti-polynomials and note their basic properties. This is followed by examples of orbit portraits of various types. The classification Theorem \ref{complete antiholomorphic} asserts that these are the only types of orbit portraits that can occur in our setting. The proof of this theorem involves a careful study of the characteristic angles of an orbit portrait and this is carried out in Section \ref{classification}. For even-periodic cycles, the first return map is holomorphic and the situation is completely similar to that of holomorphic unicritical polynomials. However, when the period of a cycle is odd, the first return map is orientation-reversing and the combinatorics of the orbit portraits associated with them turn out to be quite restricted thanks to Lemma \ref{three}, which states that at most $3$ periodic dynamical rays can land at a periodic point of odd period of a unicritical anti-polynomial. In Section \ref{formal}, we define \emph{formal orbit portraits} as a finite collection of finite subsets of $\mathbb{Q}/\mathbb{Z}$ satisfying some properties, and we prove a realization theorem in the sense that, for every formal orbit portrait, there exists a unicritical anti-polynomial (outside the multicorns) with a periodic cycle admitting the given formal orbit portrait.

\dedicatory{The author would like to thank Dierk Schleicher for his useful advice and suggestions for the improvement of the text. Thanks also go to Brennan Bell for reading the original manuscript carefully. This work was supported by a grant from the Deutsche Forschungsgemeinschaft DFG, which we gratefully acknowledge.}
\section{Orbit Portraits}
In this section, we define orbit portraits for unicritical anti-polynomials and prove some of their basic properties.  We first recall some background results regarding the dynamical rays and their dynamics.

Let $f_c = \bar{z}^d + c$ be a unicritical anti-polynomial of degree d (any unicritical anti-polynomial of degree d can be affinely conjugated to an anti-polynomial of the above form). By \cite[Lemma 1]{Na1}, there is a B\"{o}ttcher map near $\infty$ that conjugates $f_c$ to $\bar{z}^d$. Using this, one can define the dynamical rays of $f_c$ as pre-images of the radial lines. The dynamical ray $\mathcal{R}_t$ at angle $t\in \mathbb{R}/\mathbb{Z}$ maps to the dynamical ray $\mathcal{R}_{-dt}$ at angle $-dt$ under $f_c$. We refer the readers to \cite[Section 3]{NS} for more details. 

We measure angles in the fraction of a whole turn, i.e., our angles are elements of $\mathbb{S}^1 \cong \mathbb{R}/\mathbb{Z}$. We define for two different angles $\theta_1,\theta_2\in \mathbb{R}/\mathbb{Z}$, the interval $\left(\theta_1,\theta_2\right)\subset \mathbb{R}/\mathbb{Z}$ as the open connected component of $\mathbb{R}/\mathbb{Z} \setminus \lbrace \theta_1,\theta_2 \rbrace$ that consists of the angles we traverse if we move on $\mathbb{R}/\mathbb{Z}$ in counter-clockwise direction from $\theta_1$ to $\theta_2$.  Finally, we denote the length of an interval $I_1 \subset \mathbb{R}/\mathbb{Z}$ by $\ell(I_1)$ such that $\ell(\mathbb{S}_1)=1$.

\begin{definition} Let $\mathcal{O} = \lbrace z_1 , z_2 ,\cdots,z_p\rbrace$ be a periodic cycle of a unicritical anti-polynomial $f$. If a dynamical ray $\mathcal{R}_t^f$ at a rational angle $t \in \mathbb{Q}/\mathbb{Z}$ lands at some $z_i$; then for all j, the set $\mathcal{A}_j$ of the angles of all the dynamical rays landing at $z_j$ is a non-empty finite subset of $\mathbb{Q}/\mathbb{Z}$. The collection $\lbrace \mathcal{A}_1, \mathcal{A}_2, \cdots, \mathcal{A}_p \rbrace$ will be called the \emph{Orbit Portrait} $\mathcal{P(O)}$ of the orbit $\mathcal{O}$ corresponding to the anti-polynomial $f$.
\end{definition}

An orbit portrait $\mathcal{P(O)}$ will be called trivial if only one ray lands at each point of $\mathcal{O}$; i.e. $\vert \mathcal{A}_j \vert = 1, \hspace{1mm} \forall j$. From now on, we will denote an orbit portrait simply by $\mathcal{P}$; the associated orbit will be clear from the context.

\begin{lemma}[Unlinking Property]
For any orbit portrait $\displaystyle$ $\mathcal{P}$ associated with a periodic orbit of a unicritical anti-polynomial, the sets $\displaystyle \mathcal{A}_1 , \mathcal{A}_2, \cdots, \mathcal{A}_p$ are \emph{pairwise unlinked}, that is, for each $i \neq j$ the sets $\mathcal{A}_i$ and $\mathcal{A}_j$ are contained in disjoint sub-intervals of $\mathbb{R}/\mathbb{Z}$.
\end{lemma}

\begin{proof}
This follows from the fact that two rays cannot cross each other.
\end{proof}

\begin{lemma}[Orientation Reversal] For any anti-polynomial $f$, if the dynamical ray $\mathcal{R}_t$ at angle $t$ lands at a point $z \in J(f)$, then the image ray $f \left(\mathcal{R}_t\right) = \mathcal{R}_{-dt}$ lands at the point $f(z)$. Furthermore, if three or more dynamical rays land at $z$, then the cyclic order of their angles around $\mathbb{R}/\mathbb{Z}$ is reversed by the action of  $f$ ; i.e. if the rays $\lbrace\mathcal{R}_{t_1}, \mathcal{R}_{t_2},\cdots,\mathcal{R}_{t_v}\rbrace$ land at $z$, then the cyclic order of $\lbrace -dt_1, -dt_2,\cdots, -dt_v \rbrace$ is the opposite of that of $\lbrace t_1, t_2,\cdots,t_v \rbrace$.
\end{lemma}

\begin{proof}
Since the ray $\mathcal{R}_{t}$ lands at $z$, it must not pass through any pre-critical point of $f$, hence the same is true for the image ray $\mathcal{R}_{-dt}$. Therefore, the image ray is well-defined all the way to the Julia set and continuity of $f$ implies that it lands at $f(z)$.

For the second part, observe that $f$ is a local orientation-reversing diffeomorphism from $z_i$ to $z_{i+1}$. Hence, it reverses the cyclic order of the rays.
\end{proof}

\begin{lemma}[Finitely Many Rays] If a dynamical ray at a rational angle lands at some point of a periodic orbit $\mathcal{O}$ of an anti-polynomial, then only finitely many rays land at each point of $\mathcal{O}$ and all these rays are periodic.
\end{lemma}

\begin{remark} An angle $t \in \mathbb{R}/\mathbb{Z}$ (resp. a ray $\mathcal{R}_t$) is periodic under multiplication by $-d$ (resp. under $f$) if and only if $t = a/b$ (in the reduced form), for some $ a, b\in \mathbb{N}$ with $g.c.d.\left(b,d\right)=1.$ On the other hand, $t$ (resp. $\mathcal{R}_t$) is strictly pre-periodic if and only if $t = a/b$ (in the reduced form), for some $ a, b\in \mathbb{N}$ with $g.c.d.\left(b,d\right)\neq1.$
\end{remark}

\begin{proof}
Note that if $f$ is an anti-polynomial of degree $d$, then $f^{\circ 2}$ is an ordinary polynomial of degree $d^2$. Also, $f$ and $f^{\circ 2}$ have the same B\"{o}ttcher maps and $\mathcal{R}_t$ is the same curve viewed as a dynamical ray of $f$ or of $f^{\circ 2}$.

If a dynamical ray at a rational angle $t$ lands at an $f$-periodic point, then there is also a periodic (under $f$) dynamical ray $\mathcal{R}_s$ landing at that point. Clearly, $\mathcal{R}_s$ is periodic under $f^{\circ 2}$ as well. But it is well-known (see \cite[Lemma 2.3]{M2a}) for ordinary polynomials that if a periodic ray lands at a periodic point, then all rays landing there are periodic with the same period and hence there are finitely many of them. This proves the lemma.
\end{proof}

In the above lemma, we did not claim that all rays landing at a periodic point have the same period under an anti-polynomial $f$. In fact, if a periodic point (of an anti-polynomial) has odd period, rays of different periods (under multiplication by $-d$) can indeed land there. This was first proved in \cite{NS}, we include a different proof largely for the sake of completeness. 

\begin{lemma}
If a periodic orbit (of an anti-polynomial) has even length $p$ and if a rational dynamical ray lands at some point of the orbit, then all the rays landing at the periodic orbit have equal period and the common ray period can be any multiple of $p$.
\end{lemma}
\begin{proof}
Consider a periodic orbit $\lbrace z_1 , z_2 ,\cdots,z_p\rbrace$ of even period $p$. The first return map $f^{\circ p}$ is holomorphic (being an even iterate of an anti-holomorphic map). One can now argue as in \cite[Lemma 2.3]{M2a} to complete the proof.
\end{proof}

\begin{lemma}[Restricted Periods] \label{Restricted periods} Let $f$ be an anti-polynomial and $z$ be a periodic point of odd period $p$ of $f$ such that at least one periodic dynamical ray lands at $z$. Then the period of any dynamical ray landing at $z$ is either $p$ or $2p$. Moreover, the number of rays of odd period $p$ landing at $z$ is at most 2.
\end{lemma}

\begin{proof}
Let $\mathcal{O}$ be a periodic orbit of odd period $k$ for an anti-polynomial $f$ of degree $d$ such that all rays landing at $\mathcal{O}$ are periodic. Let $\theta_0\in \mathcal{A}_j$ and define $\theta_{n+1} := (-d)^p\theta_n$. We claim that $\theta_2 = \theta_0$. Otherwise, $\theta_0, \theta_1, \theta_2$ are three distinct angles. There can be two cases: $\theta_0, \theta_1, \theta_2$ lie in clockwise or in counter-clockwise order. We work with the counter-clockwise case, the other one is similar.

\begin{figure}
\centering
\includegraphics[scale=0.38]{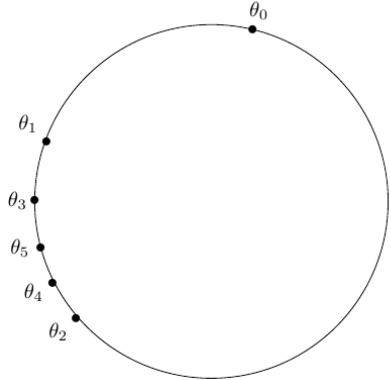}
\caption{The nesting property of the orbit of $\theta_0$ contradicts the periodicity.}
\label{nesting}
\end{figure}

Since $ \theta_0, \theta_1, \theta_2$ lie in counter-clockwise order, the orientation reversal of multiplication by $(-d)^p$ guarantees that $\theta_1, \theta_2, \theta_3$ lie in clockwise order; i.e. $\theta_3$ lies in the component of $\mathbb{R}/\mathbb{Z}\setminus \lbrace \theta_1, \theta_2\rbrace$ that doesn't contain $\theta_0$. Continuing this process inductively, one sees that $\theta_{n+1}$ belongs to the component of $\mathbb{R}/\mathbb{Z}\setminus \lbrace \theta_n, \theta_{n-1}\rbrace$ that doesn't contain $\theta_0$ (as shown in Figure \ref{nesting}). Therefore, $\theta_n \neq \theta_0, \forall n\in \mathbb{N}$: which contradicts the periodicity of $\theta_0$.

For the second part, suppose there are three angles $\theta, \theta^{\prime}, \theta^{\prime\prime} \in \mathcal{A}_j$ of period $p$ in (say) counter-clockwise order. Then multiplication by $(-d)^p$ would map them to $\lbrace\theta, \theta^{\prime}, \theta^{\prime\prime}\rbrace $ in clockwise order; an impossibility. Thus there can be at most two rays of period $p$.
\end{proof}

\subsection{Examples of Orbit Portraits}
1. Even orbit period. Let $f(z) = \bar{z}^2 - 5/4.$ $f$ admits the orbit portrait $\lbrace \lbrace 1/5 , 4/5\rbrace , \lbrace 2/5, 3/5\rbrace \rbrace$ and the angles are transitively permuted by the map $\theta \mapsto -2\theta$. 

2. Odd orbit period, transitivity. Let $f(z) = \bar{z}^2 - 7/4$. This map has a parabolic 3-cycle and all the angles of the associated orbit portrait $\displaystyle$ $\lbrace \lbrace 3/7 , 4/7 \rbrace$, $\lbrace 2/7 , 5/7 \rbrace$, $\lbrace 1/7 , 6/7 \rbrace \rbrace$ are permuted transitively by the dynamics.


3. Odd orbit period, non-transitivity. Let $f(z) = \bar{z}^2 - 1.77$. There is a periodic cycle of period 3 and the corresponding orbit portrait is $\lbrace \lbrace 1/9 , 8/9 \rbrace$, $\lbrace 2/9, 7/9 \rbrace$, $\lbrace 4/9, 5/9 \rbrace \rbrace$. All the rays are fixed by the first return map; i.e. their common ray period is 3.

4. Odd orbit period, rays with different periods. For $c \approx -1.746 + 0.008 i$, the anti-polynomial $f(z) = \bar{z}^2 + c$ has a $3$-periodic orbit with associated orbit portrait $\lbrace \lbrace 3/7, 4/9, 4/7 \rbrace , \lbrace 1/7, 1/9, 6/7 \rbrace , \lbrace 5/7, 7/9, 2/7 \rbrace \rbrace$. Note that exactly two rays of period 6 and one ray of period 3 land at each point of the orbit. 

Pictorial illustrations of various types of orbit portraits can be found in \cite[Figure 5]{NS}.
\subsection{Classification of Orbit Portraits for Unicritical Anti-polynomials}\label{classification}

So far all our discussions hold for general anti-polynomials of degree $d$. In this sub-section, we investigate the consequences of unicriticality on orbit portraits and give a complete classification of the orbit portraits that can arise in this setting. The main theorem is the following:

\begin{theorem}\label{complete antiholomorphic}
Let $f$ be a unicritical anti-polynomial of degree $d$ and $\mathcal{O} = \lbrace z_1 , z_2 ,\cdots$, $z_p\rbrace$ be a periodic orbit such that at least one rational dynamical ray lands at some $z_j$. Then the associated orbit portrait (which we assume to be non-trivial) $\mathcal{P} = \lbrace \mathcal{A}_1 , \mathcal{A}_2, \cdots, \mathcal{A}_p \rbrace$ satisfies the following properties:
\begin{enumerate}
\item Each $\mathcal{A}_j$ is a finite non-empty subset of $\mathbb{Q}/\mathbb{Z}$.

\item The map $\theta\mapsto -d\theta$ maps $\mathcal{A}_j$ bijectively onto $\mathcal{A}_{j+1}$ and reverses their cyclic order.

\item For each $j$, $\mathcal{A}_j$ is contained in some arc of length less than $1/d$ in $\mathbb{R}/\mathbb{Z}$.

\item \label{unlinked} For every $\mathcal{A}_i$, the translated sets $\mathcal{A}_{i,j}:=\mathcal{A}_i+j/d$, ($j=0,1,2,\cdots,d-1$) are unlinked from each other and from all other $\mathcal{A}_m$.

\item Each $\theta \in \mathcal{A}_j$ is periodic under $\theta \mapsto -d\theta$ and there are four possibilities for their periods: 
\begin{enumerate}

\item If $p$ is even, then all angles in $\mathcal{P}$ have the same period $rp$ for some $r\geq 1$.

\item If $p$ is odd, then one of the following three possibilities must be realized:
\begin{enumerate}
\item $\vert \mathcal{A}_j\vert = 2$ and both angles have period $p$.

\item $\vert \mathcal{A}_j\vert = 2$ and both angles have period $2p$.

\item $\vert \mathcal{A}_j\vert = 3$; one angle has period $p$ and the other two angles have period $2p$.
\end{enumerate}
\end{enumerate}
\end{enumerate}
\end{theorem}

We divide the proof in various lemmas, the next two are essentially due to Milnor, who proved them for quadratic polynomials in \cite{M2a}.

\begin{lemma}[The Critical Arc]
Let $f$ be a unicritical anti-polynomial of degree $d$ and $\mathcal{O} = \lbrace z_1 , z_2 ,\cdots,z_p\rbrace$ be an orbit of period $p$. Let $\mathcal{P} = \lbrace \mathcal{A}_1 , \mathcal{A}_2, \cdots, \mathcal{A}_p \rbrace$ be the corresponding orbit portrait. For each $j \in \lbrace 1,2,\cdots,p\rbrace$, $\mathcal{A}_j$ is contained in some arc of length less than $1/d$ in $\mathbb{R}/\mathbb{Z} $. Thus, all but one connected component of $\left(\mathbb{R}/\mathbb{Z}\right) \setminus \mathcal{A}_j $ maps bijectively to some connected component of $\left(\mathbb{R}/\mathbb{Z}\right) \setminus \mathcal{A}_{j+1}$ and the remaining complementary arc of $\left(\mathbb{R}/\mathbb{Z}\right) \setminus \mathcal{A}_j$ covers one particular complementary arc of $\mathcal{A}_{j+1}$ $d$-times and all others $(d-1)$-times.
\end{lemma}

\begin{proof}
Let $\theta \in \mathcal{A}_j$. Let $\beta$ be the element of $\mathcal{A}_j$ that lies in $\left[\theta , \theta+1/d\right)$ and is closest to $\left(\theta+1/d\right).$ Similarly, let $\alpha$ be the member of $\mathcal{A}_j$ that lies in $\left(\theta-1/d , \theta \right]$ and is closest to $\left(\theta-1/d\right).$ Note that there is no element of $\mathcal{A}_j$ in $\left(\beta , \beta+1/d\right];$ otherwise the orientation reversal property of multiplication by $-d$ would be violated. Similarly, $\left[\alpha-1/d , \alpha \right)$ contains no element of $\mathcal{A}_j$. Also, the arc $\left( \alpha , \beta \right)$ must have length less than $1/d$. We will show that the entire set $\mathcal{A}_j$ is contained in the arc $\left( \alpha , \beta \right)$ of length less than $1/d$.

\begin{figure}[!ht]
\centering
\includegraphics[scale=0.38]{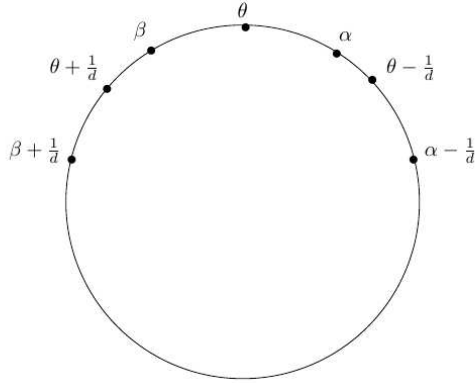}
\caption{No element of $\mathcal{A}_j$ lies outside $\left( \alpha , \beta \right)$.}
\label{critical_arc}
\end{figure}
If there exists some $\gamma \in \mathcal{A}_j$ lying outside $\left( \alpha , \beta \right)$, then $\gamma \in \left( \beta +1/d , \alpha -1/d\right)$. Therefore, there exist at least two complementary arcs of $\mathbb{R}/\mathbb{Z}\setminus \mathcal{A}_j$ of length greater than $1/d$. Both these arcs cover the whole circle and some other arc(s) of $\mathbb{R}/\mathbb{Z}\setminus \mathcal{A}_{j+1}$ under multiplication by $-d$. In the dynamical plane of $f$, the two corresponding sectors (of angular width greater than $1/d$) map to the whole plane and some other sector(s) under the dynamics. Therefore, both these sectors contain at least one critical point of $f$. This contradicts the unicriticality of $f$.

This proves that the entire set $\mathcal{A}_j$ is contained in the arc $\left( \alpha , \beta \right)$ of length less than $1/d$.
\end{proof}

\begin{remark} Following Milnor \cite{M2a}, the largest component of $\left(\mathbb{R}/\mathbb{Z}\right) \setminus \mathcal{A}_j $ (of length greater than $(1 - 1/d))$ will be called the \emph{critical arc} of $\mathcal{A}_j$ and the complementary component of $\mathcal{A}_{j+1}$ that is covered $d$-times  by the critical arc of $\mathcal{A}_j$, will be called the \emph{critical value arc} of $\mathcal{A}_{j+1}.$ In the dynamical plane of $f$, the two rays corresponding to the two endpoints of the critical arc of $\mathcal{A}_j$ along with their common landing point bound a sector containing the unique critical point of $f$. This sector is called a \emph{critical sector}. Analogously, the sector bounded by the two rays corresponding to the two endpoints of the critical value arc of $\mathcal{A}_{j+1}$ and their common landing point contains the unique critical value of $f$. This sector is called a \emph{critical value sector}. 
\end{remark}

\begin{lemma}[The Characteristic Arc] \label{blah}
Among all the complementary arcs of the various $\mathcal{A}_j$'s, there is a unique one of minimum length. It is a critical value arc for some $\mathcal{A}_j $ and is strictly contained in all other critical value arcs.
\end{lemma}

\begin{proof}
Among all the complementary arcs of the various $\mathcal{A}_j$'s, there is clearly at least one, say $\left(t^-,t^+\right)$, of minimal length. This arc must be a critical value arc for some $\mathcal{A}_j$: else it would be the diffeomorphic image of some arc of $1/d$ times its length. Let $\left(a,b\right)$ be a critical value arc for some $\mathcal{A}_k$ with $k\neq j$. Both the critical value sectors in the dynamical plane of $f$ contain the unique critical value (say, $c$) of $f$. Clearly, $\left(t^-,t^+\right) \bigcap \left(a,b\right) \neq \emptyset$. From the unlinking property of orbit portraits, it follows that $\left(t^-,t^+\right)$ and $\left(a,b\right)$ are strictly nested; i.e. the critical value sector $\left(a,b\right)$ strictly contains $\left(t^-,t^+\right)$.
\end{proof}

This shortest arc $\mathcal{I}_{\mathcal{P}}$ is called the \emph{characteristic arc} of the orbit portrait and the two angles at the ends of this arc are called the \emph{characteristic angles}. The characteristic angles, in some sense, are crucial to the understanding of orbit portraits.

In the next lemma, we investigate the special properties of orbit portraits associated with a periodic orbit of \emph{odd} period for a unicritical anti-polynomial in terms of their characteristic angles.

\begin{lemma} [Different Periods of Characteristic Angles] \label{mixed}
Let $\displaystyle$ $\mathcal{O} = \lbrace z_1, z_2,$ $\cdots, z_p\rbrace$ be a periodic orbit of a unicritical anti-polynomial $f$ with associated orbit portrait $\mathcal{P}$. If $p$ is odd and $\vert \mathcal{A}_1 \vert \geq 3,$ then one characteristic angle of $\mathcal{P}$ has period $p$ and the other has period $2p$.
\end{lemma}

\begin{proof}
Without loss of generality, we can assume that the characteristic arc $\mathcal{I}_{\mathcal{P}}$ is a critical value arc of $\mathcal{A}_1$. Since $\vert \mathcal{A}_1 \vert \geq 3,$ $\mathcal{A}_1$ has at least three complementary components. Let $\mathcal{I}^{+}$ be the arc just to the right of $\mathcal{I}_{\mathcal{P}}$ and $I^{-}$ be the one just to the left of $\mathcal{I}_{\mathcal{P}}$. We can also assume that $\mathcal{I}^{-}$ no shorter than $\mathcal{I}^{+}$; i.e. $\textit{l}\left(\mathcal{I}^{-}\right) \geq \textit{l}\left(\mathcal{I}^{+}\right).$ Since $\mathcal{I}^{+}$ is not the critical value arc of $\mathcal{A}_1$, there must exist a critical value arc $\mathcal{I}_c$ which maps diffeomorphically onto $I^{+}$ under some iterate of multiplication by $-d$; i.e. $\mathcal{I}^{+} = \left( -d \right)^m \mathcal{I}_c$, for some $m \geq 1$.

We claim that $\mathcal{I}_c = \mathcal{I}_{\mathcal{P}}$. Otherwise, $\mathcal{I}_c$ would properly contain the characteristic arc $\mathcal{I}_{\mathcal{P}}$. Since $\mathcal{I}_c$ is strictly smaller than $\mathcal{I}^{+}$, $\mathcal{I}_c$ cannot contain $\mathcal{I}^{+}$. So one end of $\mathcal{I}_c$ must lie in $\mathcal{I}^{+}$; but then it follows from the unlinking property that both ends of $\mathcal{I}_c$ are in $\mathcal{I}^{+}$. Therefore, $\mathcal{I}_c$ strictly contains $\mathcal{I}^{-}$. But this is impossible because $\textit{l}\left(\mathcal{I}^{-}\right) \geq \textit{l}\left(\mathcal{I}^{+}\right) > \textit{l}\left(\mathcal{I}_{c}\right). $

Therefore, $\mathcal{I}^{+} = \left( -d \right)^m \mathcal{I}_{\mathcal{P}}$. Note that $m$ must be a proper multiple of $p$, say $m=pq$, for some $q \geq 1$. Also let, $\mathcal{I}^{+} = \left( a , b \right)$ and $\mathcal{I}_{\mathcal{P}} = \left( b , c \right)$.  Now we consider two cases:

\emph{Case I. $q$ is even.} Since $\mathcal{I}_{\mathcal{P}}$ maps to $\mathcal{I}^{+}$ by an orientation preserving diffeomorphism, we have: $b = d^{pq} c$ and $a = d^{pq} b$. So, $a = \left(d^{2p}\right)^q c = c$ (for an odd-periodic orbit, every ray is fixed by the second return map). This contradicts the fact that $\vert \mathcal{A}_1 \vert \geq 3.$ So $q$ must be odd.

\begin{figure}[!ht]
\centering
\includegraphics[scale=0.38]{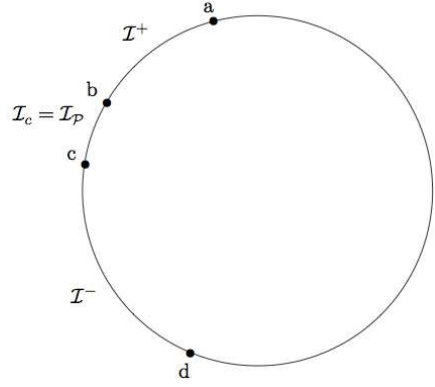}
\caption{The characteristic arc $\mathcal{I}_{\mathcal{P}}$ maps to the shorter adjacent arc $\mathcal{I}^{+}$.}
\label{mixed1}
\end{figure}

\emph{Case II. $q$ is odd.} In this case, the arc $\left( b , c \right)$ maps to $\left( a , b \right)$ by an orientation reversing diffeomorphism. Hence, $b = \left(-d\right)^{pq} b$ and $a = \left(-d\right)^{pq} c.$  Since $b$ is fixed by an odd iterate of the first return map, it must have period $p$. On the other hand, $a$ and $c$ are two distinct angles such that $a$ belongs to the orbit of $c$ under the first return map of $\mathcal{A}_1$. Since every angle is fixed by the second return map, it follows that $\lbrace a ,c \rbrace$ is a 2-cycle under multiplication by $\left(-d\right)^p$; hence they both have period $2p$.
\end{proof}

The following lemma vastly limits the possibilities of orbit portraits for an odd-periodic cycle.

\begin{lemma}[No More Than Three Rays] \label{three}
If $f$ is a unicritical anti-polynomial, then at most three periodic dynamical rays can land at a periodic point of odd period of $f$. 
\end{lemma}
 
\begin{proof}
Let  $\mathcal{O} = \lbrace z_1 , z_2 ,\cdots,z_p\rbrace$ be an orbit of odd period $p$ with associated orbit portrait $\mathcal{P} = \lbrace \mathcal{A}_1 , \mathcal{A}_2, \cdots, \mathcal{A}_p \rbrace$. Suppose more than three periodic rays land at $z_1$ and the characteristic arc $\mathcal{I}_{\mathcal{P}}$ is the critical value arc of $\mathcal{A}_1$. Then $\mathcal{A}_1$ has at least three complementary components other than $\mathcal{I}_{\mathcal{P}}$: $\mathcal{I}_{1}, \mathcal{I}_{2}, \mathcal{I}_{3}$  and we can assume that $\textit{l}\left(\mathcal{I}_{1}\right) \geq \textit{l}\left(\mathcal{I}_{2}\right) \geq \textit{l}\left(\mathcal{I}_{3}\right)$.

Since $\mathcal{I}_{2}$ is not a critical value arc, some critical value arc $\mathcal{I}_c$ must map diffeomorphically onto $\mathcal{I}_{2}$ under some iterate of multiplication by $-d$. Thus $\mathcal{I}_c$ is strictly smaller than $\mathcal{I}_{2}.$ We again claim that $\mathcal{I}_c = \mathcal{I}_{\mathcal{P}}$. If not, then $\mathcal{I}_c$ strictly contains $\mathcal{I}_{\mathcal{P}}.$ At least one end of $\mathcal{I}_c$ must lie in $\mathcal{I}_{2};$ else 
$\mathcal{I}_c$ would contain $\mathcal{I}_{2},$ but $\textit{l}\left(\mathcal{I}_{2}\right) > \textit{l}\left(\mathcal{I}_c\right).$ It follows from the unlinking property that both ends of $\mathcal{I}_c$ lies in $\mathcal{I}_{2}.$ But then $\mathcal{I}_c$ would contain $\mathcal{I}_{1}$, which is impossible as $\textit{l}\left(\mathcal{I}_{1}\right) \geq \textit{l}\left(\mathcal{I}_{2}\right) > \textit{l}\left(\mathcal{I}_{c}\right).$ Therefore, $\mathcal{I}_{2} = \left( -d \right)^{m_1} \mathcal{I}_{\mathcal{P}},$ with $m_1 = pr_1$ for some $r_1 \geq 1$.

Applying the same argument on $\mathcal{I}_{3}$, we have: $\mathcal{I}_{3} = \left( -d \right)^{m_2} \mathcal{I}_{\mathcal{P}},$ with $m_2 = pr_2$ for some $r_2 \geq 1$.

But from the previous lemma, the orbits of the end-points of $\mathcal{I}_{\mathcal{P}}$ under multiplication by $\left( -d \right)^{p}$ consist of only three points: 
the end-points of $\mathcal{I}_{\mathcal{P}}$ itself and those of one of its adjacent arcs. This contradicts the fact that both $\mathcal{I}_{2}$ and $\mathcal{I}_{3}$ are different from $\mathcal{I}_{\mathcal{P}}$ and finishes the proof.
\end{proof}

\begin{remark}
An alternative proof of Lemma \ref{three} follows from \cite[Theorem 5.2]{Ki}: for a unicritical anti-polynomial $f$, the second iterate $f^{\circ 2}$ has exactly $2$ critical values and hence, any periodic orbit portrait for $f^{\circ 2}$ can have at most $3$ cycles of rays. However, Kiwi's theorem does not provide any information about how these dynamical rays are permuted by the anti-polynomial $f$. In particular, the exact periods (under $f$) of the characteristic angles and their relative positions are important for our applications and these are discussed in Lemma \ref{mixed} and Lemma \ref{lengths}.
\end{remark}
We are now in a position to prove the main theorem of this section.

\begin{proof}[Proof of Theorem \ref{complete antiholomorphic}]
Property (\ref{unlinked}) simply states the fact that if two periodic rays $\mathcal{R}_{\theta}^c$ and $\mathcal{R}_{\theta'}^c$ land together at some periodic point $z$, then the rays $\mathcal{R}_{\theta+j/d}^c$ and $\mathcal{R}_{\theta'+j/d}$ land together at a pre-periodic point $z'$ with $f(z')=f(z)$, since the Julia set of $f$ has a $d$-fold rotation symmetry. The rest of the properties follow from the previous lemmas. 
\end{proof}

\section{Formal Orbit Portraits And A Realization Theorem}\label{formal}

\begin{definition}
A finite collection $\mathcal{P} = \lbrace \mathcal{A}_1 , \mathcal{A}_2, \cdots, \mathcal{A}_p \rbrace$ of subsets of $\mathbb{R}/\mathbb{Z}$ satisfying the five properties of Theorem \ref{complete antiholomorphic} is called a \emph{formal orbit portrait}. The property (4) of Theorem \ref{complete antiholomorphic} implies that each $\mathcal{A}_j$ has a complementary arc of length greater than $\left( 1 - 1/d\right)$ (which we call the critical arc of  $\mathcal{A}_j$) that, under multiplication by $-d$ covers exactly one complementary arc of $\mathcal{A}_{j+1}$ d-times (which we call the critical value arc of $\mathcal{A}_{j+1}$) and the others $\left( d - 1\right)$-times. 
\end{definition}

The principal goal of this section is to prove the following realization theorem.

\begin{theorem}\label{main}
Let $\mathcal{P} = \lbrace \mathcal{A}_1 , \mathcal{A}_2, \cdots, \mathcal{A}_p \rbrace$ be a formal orbit portrait. Then there exists some $c \in \mathbb{C} \setminus \mathcal{M}^{\ast}_d$, such that $f(z)= \bar{z}^d+ c$ has a repelling periodic orbit with associated orbit portrait $\mathcal{P}$.
\end{theorem}

We borrow many techniques from the proof of the realization of formal orbit portraits for quadratic polynomials given in \cite{M2a}. The anti-holomorphic case is slightly more difficult than the holomorphic one and the proof will be divided into several cases. 

The next lemma is a combinatorial version of Lemma \ref{blah} and this is where condition (\ref{unlinked}) of the definition of formal orbit portraits comes in.

\begin{lemma}\label{sixth condition}
Let $\mathcal{P} = \lbrace \mathcal{A}_1 , \mathcal{A}_2, \cdots, \mathcal{A}_p \rbrace$ be a formal orbit portrait. Among all the complementary arcs of the various $\mathcal{A}_j$'s, there is a unique one of minimum length. It is a critical value arc for some $\mathcal{A}_j $ and is strictly contained in all other critical value arcs.
\end{lemma}

\begin{proof}
Among all the complementary arcs of the various $\mathcal{A}_j$'s, there is clearly at least one, say $\mathcal{I} =  \left(t^-,t^+\right)$, of minimal length $l$. This arc must be a critical value arc of some $\mathcal{A}_j$,  else it would be the diffeomorphic image of some arc of $1/d$ times its length. Let $\mathcal{I^{\prime}} = \left(a,b\right)$ be the critical arc of $\mathcal{A}_{j-1}$ having length $\{(d-1)+l\}/d$ so that its image, under multiplication by $-d$, covers $\left(t^-,t^+\right)$ $d$-times and the rest of the circle exactly $(d-1)$-times. $\left(t^-,t^+\right)$ has $d$ pre-images $-\mathcal{I}/d$, $\left(-\mathcal{I}/d+1/d\right)$, $\left(-\mathcal{I}/d+2/d\right), \cdots, \left(-\mathcal{I}/d+(d-1)/d\right)$ (as shown in Figure \ref{complicated}); each of them is contained in $\left(a,b\right)$ and has length $l/d$. By our minimality assumption, $\left(t^-,t^+\right)$ contains no angle of $\mathcal{P}$ and hence neither do its d pre-images. Label the d connected components of $\displaystyle \mathbb{R}/\mathbb{Z} \setminus \bigcup_{r=0}^{d-1} \left(-\mathcal{I}/d+r/d\right)$ as $C_1, C_2,\cdots, C_d$ with $C_1 = \lbrack b,a \rbrack$.

Clearly, $\mathcal{A}_{j-1}$ is contained in $C_1$ and the two end-points $a$ and $b$ of $C_1$ belong to $\mathcal{A}_{j-1}$. Also, $C_{i+1} = C_1 + i/d$ for $0\leq i \leq d-1$. Therefore, $\mathcal{A}_{j-1}+i/d$ is contained in $C_{i+1}$ with the end-points of $C_{i+1}$ belonging to $\mathcal{A}_{j-1}+i/d$. By condition (\ref{unlinked}) of the definition of formal orbit portraits, each $\mathcal{A}_{j-1}+i/d$ (for fixed $j$ and varying $i$) is unlinked from $\mathcal{A}_{k}$, for $k\neq j-1$. This implies that for any $k \neq j-1$, $\exists !$ $r_k \in \lbrace 1, 2, \cdots, d \rbrace$ such that $\mathcal{A}_k$  is contained in int$ (C_{r_k})$. Hence, all the non-critical arcs of $\mathcal{A}_k$ would be contained in the interior of $C_{r_k}$. Thus all the non-critical value arcs of $\mathcal{A}_{k+1}$ $(k+1 \neq j)$ are contained $\displaystyle \mathbb{R}/\mathbb{Z} \setminus \left[t^-,t^+\right]$. Hence the critical value arc of any $\mathcal{A}_{m}$ $(m \neq j)$ strictly contains $\mathcal{I} =  \left(t^-,t^+\right)$. The uniqueness follows.
\end{proof}

We will also need an anti-holomorphic analogue of a classical result about unicritical polynomials (see \cite[Lemma 2.7]{M2a}, \cite[Lemma 2.4]{S1a}).

\begin{lemma}\label{folklore}
For a formal anti-holomorphic orbit portrait $\displaystyle$ $\mathcal{P} = \lbrace \mathcal{A}_1 , \mathcal{A}_2, \cdots,$ $\mathcal{A}_{p} \rbrace$ with even $p$, multiplication by $-d$ either permutes all the angles of $\mathcal{P}$ transitively or $\vert \mathcal{A}_j \vert = 2 \hspace{1mm} \forall \hspace{1mm} j$ and the first return map of $\mathcal{A}_j$ fixes each angle.
\end{lemma}

\begin{proof}
We assume that the cardinality of each $\mathcal{A}_j$ is at least three and we'll show that multiplication by $-d$ permutes all the angles of $\mathcal{P}$. We can also assume that the characteristic arc $\mathcal{I}_{\mathcal{P}}$ is a critical value arc of $\mathcal{A}_1$. Since $\vert \mathcal{A}_1 \vert \geq 3,$ $\mathcal{A}_1$ has at least three complementary components. Let $\mathcal{I}^{+}$ be the arc just to the right of $\mathcal{I}_{\mathcal{P}}$ and $\mathcal{I}^{-}$ be the one just to the left of $\mathcal{I}_{\mathcal{P}}$. Let $\mathcal{I}^{-}$ be longer than $\mathcal{I}^{+}$; i.e. $\textit{l}\left(\mathcal{I}^{-}\right) \geq \textit{l}\left(\mathcal{I}^{+}\right).$ Since $\mathcal{I}^{+}$ is not the critical value arc of $\mathcal{A}_1$, there must exist a critical value arc $\mathcal{I}_c$ which maps diffeomorphically onto $\mathcal{I}^{+}$ under some iterate of multiplication by $-d$; i.e. $\mathcal{I}^{+} = \left( -d \right)^m \mathcal{I}_c$, for some $m \geq 1$. Arguing as in Lemma \ref{mixed}, we see that $\mathcal{I}_c = \mathcal{I}_{\mathcal{P}}$. 

Therefore, $\mathcal{I}^{+} = \left( -d \right)^m \mathcal{I}_{\mathcal{P}}$. Note that $m$ must be multiple of $p$, thus $m$ is even. Also let, $\mathcal{I}^{+} = \left( a , b \right)$ and $\mathcal{I}_{\mathcal{P}} = \left( b , c \right)$. Since $\mathcal{I}_{\mathcal{P}}$ maps to $\mathcal{I}^{+}$ by an orientation preserving diffeomorphism, we have: $b = d^{m} c$ and $a = d^{m} b$. Multiplication by $d^m$ is an orientation preserving map and it sends $\mathcal{A}_1$ bijectively onto itself such that the point $b$ is mapped to an adjacent point $a$. It follows that multiplication by $d^m$ acts transitively on $\mathcal{A}_1$. Hence multiplication by $-d$ permutes all the angles of $\mathcal{P}$ transitively.
\end{proof}

The next lemma is the key to the proof of the realization theorem and gives a necessary condition for the dynamical rays (of a unicritical anti-polynomial) at the characteristic angles of a formal orbit portrait to land at a common point.

\begin{lemma}[Outside The Multicorns]\label{Outside_The_Multicorns}
Let $\mathcal{P}$ be a formal orbit portrait and $\left(t^-,t^+\right)$ be its characteristic arc. For some  $c \notin \mathcal{M}_d^*$, the two dynamical rays $\mathcal{R}_{t^-}^{c}$ and $\mathcal{R}_{t^+}^{c}$ land at the same point of $J(f_c)$ if the external angle $t(c) \in \left(t^-,t^+\right)$. 
\end{lemma}

\begin{proof}
We retain the terminology of Lemma \ref{sixth condition}. If $c \notin \mathcal{M}_d^*$, all the periodic points of $f_c = \bar{z}^d + c$ are repelling and the Julia set is a cantor set.

Let the external angle of $c$ in the parameter plane be $t(c)$. Label the connected components of $\mathbb{R}/\mathbb{Z}\setminus \lbrace -t(c)/d , -t(c)/d + 1/d, \cdots, -t(c)/d + (d-1)/d\rbrace$ counter-clockwise as $L_0, L_1, \cdots, L_{d-1}$ such that the component containing the angle $0$ gets label $L_0$. The $t(c)$-itinerary of an angle $\theta \in \mathbb{R}/\mathbb{Z}$ is defined as a sequence $\left(a_n\right)_{n \geq 0}$ in $\lbrace 0 , 1, \cdots, d-1 \rbrace^{\mathbb{N}}$ such that $a_n = i$ if $(-d)^n\theta \in L_i$. All but countably many $\theta$'s (the ones which are not the iterated pre-images of $t(c)$ under multiplication by $-d$) have a well-defined $t(c)$-itinerary. 

\begin{figure}[ht]
\centering
\includegraphics[scale=0.32]{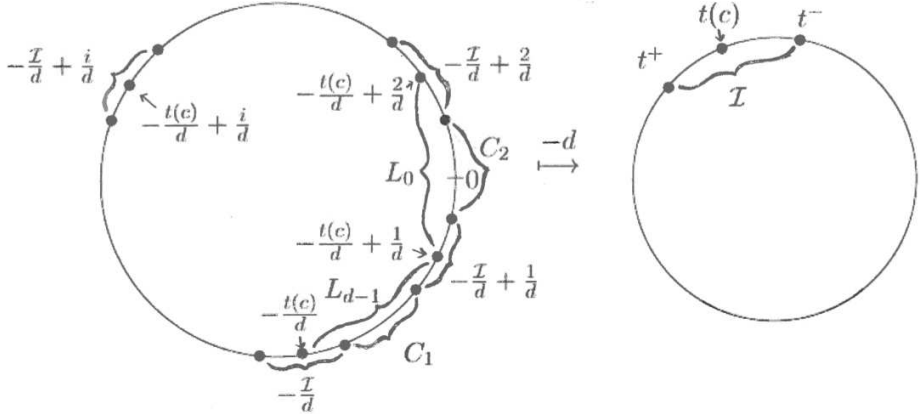}
\caption{The $d$ pre-images of $\mathcal{I} = \left( t^- , t^+ \right)$ under multiplication by $-d$ and the complementary arcs $C_i$'s are labelled on the circle. Also the $d$ pre-images of $t(c) \left(\in \left( t^- , t^+ \right)\right)$ are marked and the components $L_i$'s of $\mathbb{R}/\mathbb{Z}\setminus \lbrace -t(c)/d , -t(c)/d + 1/d, \cdots, -t(c)/d + (d-1)/d\rbrace$ are labelled. Each $C_i$ is contained in some $L_j$.}
\label{complicated}
\end{figure}

Similarly, in the dynamical plane of $f_c$, the $d$ dynamical rays $\mathcal{R}_{-t(c)/d}^{f_c}$, $\mathcal{R}_{-t(c)/d + 1/d}^{f_c},$ $\cdots$, $\mathcal{R}_{-t(c)/d + (d-1)/d}^{f_c}$ land at the critical point $0$ and cut the dynamical plane into $d$ sectors. Label these sectors counter-clockwise as $L^{\prime}_0, L^{\prime}_1, \cdots, L^{\prime}_{d-1}$ such that the component containing the dynamical ray $\mathcal{R}_{0}^{f_c}$ at angle $0$ gets label $L^{\prime}_0$. Any point $z \in J(f_c)$ has an associated symbol sequence $\left(a_n\right)_{n \geq 0}$ in $\lbrace 0 , 1, \cdots, d-1 \rbrace^{\mathbb{N}}$ such that $a_n = i$ if $f_c^{\circ n}(z) \in L^{\prime}_i$. Clearly, a dynamical ray $\mathcal{R}_{\theta}^{f_c}$ at angle $\theta$ lands at $z$ iff the $t(c)$-itinerary of $\theta$ coincides with the symbol sequence of $z$ defined above.

If $t(c) \in \mathcal{I} = \left(t^-,t^+\right)$, the $d$ angles $\lbrace-t(c)/d , -t(c)/d + 1/d, \cdots, -t(c)/d + (d-1)/d\rbrace$ lie in the $d$ intervals $-\mathcal{I}/d$, $\left(-\mathcal{I}/d+1/d\right), \cdots, \left(-\mathcal{I}/d+(d-1)/d\right)$ respectively and no element of $\mathcal{P}$ belongs to $\displaystyle \bigcup_{j=0}^{d-1} \left(-\mathcal{I}/d+j/d\right)$. First note that the rays $\mathcal{R}_{t^-}^{f_c}$ and $\mathcal{R}_{t^+}^{f_c}$ indeed land as $t(c) \notin$ the finite sets $\lbrace  t^{\pm}, -dt^{\pm}, (-d)^2t^{\pm},\cdots \rbrace$. Each $\mathcal{A}_j$ is contained in a unique $C_r.$ Therefore, for each $n\geq 0$, the angles $(-d)^n t^-$ and $(-d)^n t^+$ belong to the same $L_i$. So $t^-$ and $t^+$ have the same $t(c)$-itinerary; which implies that the two dynamical rays $\mathcal{R}_{t^-}^{f_c}$ and $\mathcal{R}_{t^+}^{f_c}$ land at the same point of $J(f_c).$
\end{proof}

In the next lemma, we take a closer look at orbit portraits associated with odd-periodic cycles such that exactly three rays land at each point of the cycle. It follows from Lemma \ref{mixed} that two of these angles must have period $2p$ and the other one has period $p$ such that the two characteristic angles have different periods.   

\begin{lemma}\label{lengths}
Let $\mathcal{P} = \lbrace \mathcal{A}_1, \mathcal{A}_2, \cdots, \mathcal{A}_p \rbrace$ be a formal orbit portrait with $p$ odd and $\vert \mathcal{A}_1 \vert =3$. Assume further that $ \mathcal{A}_1 = \lbrace t^-, t, t^+ \rbrace $ with $t^-$ and $t^+$ of period $2k$ and $t$ of period $k$ such that the characteristic angles are $\lbrace t^-, t \rbrace$. Then,
\begin{enumerate}
\item $t$ belongs to the shorter complementary component of $\mathbb{R}/\mathbb{Z} \setminus \lbrace t^-, t ^+\rbrace$,

\item $\ell$(the shorter complementary component of $\mathbb{R}/\mathbb{Z} \setminus \lbrace t^-, t ^+\rbrace$) = $\left( 1+d^p\right) \cdot \ell$(the characteristic arc of $\mathcal{P}$).
\end{enumerate}
\end{lemma}

\begin{proof}
Let $\mathcal{I}_\mathcal{P}$ be the characteristic arc of $\mathcal{P}$ with end-points $t^-$, $t$ and the other two complementary components of $\mathcal{A}_1$ be $\mathcal{I}_1$ and $\mathcal{I}_2$. Without loss of generality, we can assume that $\mathcal{I}_2$ is a critical arc. So, $\ell \left( \mathcal{I}_2 \right) > \ell \left( \mathcal{I}_1 \right) > \ell \left( \mathcal{I}_\mathcal{P} \right)$. An argument similar to that in Lemma \ref{three} shows that $\mathcal{I}_\mathcal{P}$ must map diffeomorphically onto $\mathcal{I}_{1}$ under some iterate of multiplication by $-d$; i.e. $\mathcal{I}_{1} = (-d)^{pq} \mathcal{I}_\mathcal{P}$, for some $q \in \mathbb{N}$. Since one end-point (namely, $t$) of $\mathcal{I}_\mathcal{P}$ has period $p$, it follows that $t$ is an end-point of $\mathcal{I}_{1}$ as well and $q=1$. Hence, the end-points of the critical arc $\mathcal{I}_2$ are $t^-$ and $t^+$. This proves that $t$ belongs to the shorter complementary component of $\mathbb{R}/\mathbb{Z} \setminus \lbrace t^-, t ^+\rbrace$ and $\ell \left( \mathcal{I}_{1} \right) = d^p \cdot \ell \left( \mathcal{I}_\mathcal{P} \right)$. Since the shorter complementary component of $\mathbb{R}/\mathbb{Z} \setminus \lbrace t^-, t ^+\rbrace = \mathcal{I}_\mathcal{P} \cup \mathcal{I}_{1}$, part (ii) follows.
\end{proof}

\begin{proof}[Proof Of Theorem \ref{main}]
Let $\mathcal{P}$ be a formal anti-holomorphic orbit portrait. We consider the following cases.
\begin{flushleft}
\underline{Case 1. $p$ is even} 
\end{flushleft}
First let's assume that $\vert \mathcal{A}_j \vert = 2$ and the first return map (multiplication by $(-d)^p = d^p$) of $\mathcal{A}_j$ fixes each angle. Let $\mathcal{I}_{\mathcal{P}}=\left(t^-, t^+\right)$ be the characteristic arc of the formal orbit portrait $\mathcal{P} = \lbrace \mathcal{A}_1 , \mathcal{A}_2, \cdots, \mathcal{A}_{p} \rbrace$ such that $\lbrace t^-, t^+ \rbrace \subset \mathcal{A}_1$. Choose $c$ outside $\mathcal{M}_d^*$ with $t(c) \in \left(t^-,t^+\right)$. Then, the two dynamical rays $\mathcal{R}_{t^-}^{f_c}$ and $\mathcal{R}_{t^+}^{f_c}$ land at the same point $z \in J\left(f_c\right)$. Clearly, $z$ is a periodic point of period $p^{\prime} \vert p$. Let, $\mathcal{P}^{\prime} = \lbrace \mathcal{A}_1^{\prime} , \mathcal{A}_2^{\prime}, \cdots, \mathcal{A}_{p^{\prime}}^{\prime} \rbrace$ be the orbit portrait associated with $\mathcal{O}(z)$ such that $\mathcal{A}_1^{\prime}$ is the set of angles of the rays landing at $z$. Since the two elements $t^-$ and $t^+$ of $\mathcal{A}_1^{\prime}$ are not in the same cycle under multiplication by $-d$, Lemma \ref{folklore} implies that exactly two rays land at $z$ and $ p = p^{\prime}$. Therefore, $\mathcal{P}^{\prime} = \mathcal{P}$.

On the other hand, if multiplication by $-d $ acts transitively on $\mathcal{P}$, there exists $l \in \mathbb{N}$ such that $d^{lp} t^- = t^+$.  Let $r^{\prime}$ rays land at each point of $\mathcal{O}(z)$ as above ($z$ is the common landing point of $\mathcal{R}_{t^-}^{f_c}$ and $\mathcal{R}_{t^+}^{f_c}$, where $c$ lies outside $\mathcal{M}_d^*$ with $t(c) \in \left(t^-,t^+\right)$) so that the associated orbit portrait $\mathcal{P}^{\prime}$ contains a total of $r^{\prime}p^{\prime}$ rays, where $p^{\prime}$ is the period of $z$. Lemma \ref{folklore} implies that multiplication by $-d$ acts transitively on $\mathcal{P}^{\prime}$ and the common period of all the angles in $\mathcal{P}^{\prime}$ is $r^{\prime}p^{\prime}$. Since $t^-$ and $t^+$ are adjacent angles in $\mathcal{A}_1$, it easily follows that multiplication by $d^{lp}$ acts transitively on $\mathcal{A}_1$. Therefore, all the angles in $\mathcal{A}_1$ land at $z$; i.e. $r^{\prime} \geq r$. Since $r^{\prime}p^{\prime} = rp$, we have $p^{\prime} \leq p$. If $p^{\prime}$ was strictly smaller than $p$, both the sets $\mathcal{A}_1$ and $\mathcal{A}_{1+p^{\prime}} $ would be contained in $\mathcal{A}_1^{\prime}$; hence multiplication by $d^p$ would map these two sets onto themselves preserving their cyclic order. This forces multiplication by $d^p$ to be the identity map on $\mathcal{A}_1$: a contradiction to the transitivity assumption. Thus, $p^{\prime} = p$ and $r^{\prime} = r$. Therefore, $\mathcal{A}_1 = \mathcal{A}_1^{\prime}$ and $\mathcal{P}^{\prime} = \mathcal{P}$.

\begin{flushleft}
\underline{Case 2. $p$ is odd, each $\mathcal{A}_j$ consists of exactly two angles of period $p$}
\end{flushleft}

Let $\mathcal{I}_{\mathcal{P}} = \left(t^-,t^+\right)$ be the characteristic arc of the formal orbit portrait $\mathcal{P} = \lbrace \mathcal{A}_1 , \mathcal{A}_2, \cdots, \mathcal{A}_{p} \rbrace$ such that $\mathcal{A}_1 = \lbrace t^-,t^+ \rbrace$. Choose $c$ outside $\mathcal{M}_d^*$ with $t(c) \in \left(t^-,t^+\right)$. Then, the two dynamical rays $\mathcal{R}_{t^-}^{f_c}$ and $\mathcal{R}_{t^+}^{f_c}$ (of period $k$) land at the same point $z \in J\left(f_c\right)$. Clearly, $z$ is a periodic point of period $p^{\prime} \vert p$. Since $z$ has odd period $p^{\prime}$, the periods of rays landing there can be either $p^{\prime}$ or $2p^{\prime}$. It follows that $ p = p^{\prime}$. Also, by Theorem \ref{complete antiholomorphic}, if two rays of period $p$ land at a periodic point of odd period $p$, then these are the only rays landing there. Therefore, if $\mathcal{P}^{\prime} = \lbrace \mathcal{A}_1^{\prime} , \mathcal{A}_2^{\prime}, \cdots, \mathcal{A}_{p}^{\prime} \rbrace$ is the associated orbit portrait of $\mathcal{O}(z)$ such that $\mathcal{A}_1^{\prime}$ is the set of angles of the rays landing at $z$, then $\mathcal{A}_1 = \mathcal{A}_1^{\prime}$. Hence, $\mathcal{P}^{\prime} = \mathcal{P}$.

\begin{flushleft}
\underline{Case 3. $p$ is odd, each $\mathcal{A}_j$ consists of exactly two angles of period $2p$}
\underline{and one angle of period $p$}
\end{flushleft}

Let $\mathcal{I}_{\mathcal{P}}=\left(t^-,t^+\right)$ be the characteristic arc of the formal orbit portrait $\mathcal{P} = \lbrace \mathcal{A}_1 , \mathcal{A}_2, \cdots, \mathcal{A}_{p} \rbrace$ such that $\lbrace t^-,t^+ \rbrace \subset \mathcal{A}_1$. By Lemma \ref{mixed}, one characteristic angle has period $p$ and the other has period $2p$. Choose $c$ outside the $\mathcal{M}_d^*$ with $t(c) \in \left(t^-,t^+\right)$. Then, the two dynamical rays $\mathcal{R}_{t^-}^{f_c}$ and $\mathcal{R}_{t^+}^{f_c}$ land at the same point $z \in J\left(f_c\right)$. Clearly, $z$ is a periodic point of period $p^{\prime} \vert p$. Since $z$ has odd period $p^{\prime}$, the periods of rays landing there can be either $p^{\prime}$ or $2p^{\prime}$. It follows that $ p = p^{\prime}$. Let $\mathcal{P}^{\prime} = \lbrace \mathcal{A}_1^{\prime}, \mathcal{A}_2^{\prime}, \cdots, \mathcal{A}_{p}^{\prime} \rbrace$ be the associated orbit portrait of $\mathcal{O}(z)$ such that $\mathcal{A}_1^{\prime}$ is the set of angles of the rays landing at $z $. By Theorem \ref{complete antiholomorphic}, if one ray of period $p$ and another ray of period $2p$ land at a periodic point of odd period $p$, then exactly three rays land there. Clearly, $\mathcal{A}_1 = \mathcal{A}_1^{\prime}$. Hence, $\mathcal{P}^{\prime} = \mathcal{P}$.

\begin{flushleft}
\underline{Case 4. $p$ is odd, each $\mathcal{A}_j$ consists of exactly two angles of period $2p$}
\end{flushleft}

As above, let $\mathcal{I}_{\mathcal{P}} = \left(t^-,t^+\right)$ be the characteristic arc of the formal orbit portrait $\mathcal{P} = \lbrace \mathcal{A}_1 , \mathcal{A}_2, \cdots, \mathcal{A}_{p} \rbrace$ such that $\mathcal{A}_1 = \lbrace t^-,t^+ \rbrace$. Choose $c$ outside $\mathcal{M}_d^*$ with $t(c) \in \left(t^-,t^+\right)$. Then, the two dynamical rays $\mathcal{R}_{t^-}^{f_c}$ and $\mathcal{R}_{t^+}^{f_c}$ (of period $2p$) land at the same point $z \in J\left(f_c\right)$. Clearly, $z$ is a periodic point of period some $p^{\prime}$. As $\mathcal{R}_{t^+}^{f_c}$ lands at $z$, $\mathcal{R}_{t^-}^{f_c} = f^{\circ p}\left( \mathcal{R}_{t^+}^{f_c}\right)$ must land at $f^{\circ p}\left(z\right)$. So, $f^{\circ p}\left(z\right) = z$. Hence, $p^{\prime} \vert p.$ Since $z$ has odd period $p^{\prime}$, the periods of rays landing there can be either $p^{\prime}$ or $2p^{\prime}$. It follows that $ p = p^{\prime}$. Let $\mathcal{P}^{\prime} = \lbrace \mathcal{A}_1^{\prime}, \mathcal{A}_2^{\prime}, \cdots, \mathcal{A}_{p}^{\prime} \rbrace$ be the associated orbit portrait of $\mathcal{O}(z)$ with $\mathcal{A}_1^{\prime}$ being the set of angles of the rays landing at $z$. By Theorem \ref{complete antiholomorphic}, if two rays of period $2p$ land at a periodic point of odd period $p$, then either these are the only rays landing there or there can be at most one further ray of period $p$ landing there. In the first case, $\mathcal{A}_1 = \mathcal{A}_1^{\prime}$ and thus, $\mathcal{P}^{\prime} = \mathcal{P}$.

In the second case, we must work a little harder. Let $\mathcal{I}_{\mathcal{P}^{\prime}}$ be the characteristic arc of the orbit portrait $\mathcal{P}^{\prime}$ associated with the orbit $\mathcal{O}(z)$. Since exactly three rays (two of period $2p$ and one of period $p$) land at $z$, Lemma \ref{mixed} tells that one characteristic angle has period $p$ and the other has period $2p$. We claim that the characteristic angles of $\mathcal{P}^{\prime}$ belong to $\mathcal{A}_1^{\prime}$. Let's assume that this is false and we'll establish a contradiction. Note that $t(c) \in \mathcal{I}_{\mathcal{P}^{\prime}}$ ($\mathcal{I}_{\mathcal{P}^{\prime}}$ is the characteristic arc of an actual orbit portrait for $f_c$, hence it must be a critical value arc and thus contains the external angle of the critical value $c$). So $\mathcal{I}_{\mathcal{P}^{\prime}}$ and $\mathcal{I}_{\mathcal{P}}$ intersect and hence must be strictly nested (unlinking property); in particular, $\mathcal{I}_{\mathcal{P}}$ strictly contains $\mathcal{I}_{\mathcal{P}^{\prime}}$ (since $\mathcal{I}_{\mathcal{P}^{\prime}}$ is the characteristic arc of $\mathcal{P}^{\prime}$). But one end of $\mathcal{I}_{\mathcal{P}^{\prime}}$ is an angle of period $2p$ which is already contained in $\mathcal{P}$, thus the characteristic arc $\mathcal{I}_{\mathcal{P}}$ of the formal orbit portrait $\mathcal{P}$ contains an element of $\mathcal{P}$: a contradiction. Hence, the characteristic angles of $\mathcal{P}^{\prime}$ belong to $\mathcal{A}_1^{\prime}$.

\begin{figure}[!ht]
\begin{center}
\includegraphics[scale=0.5]{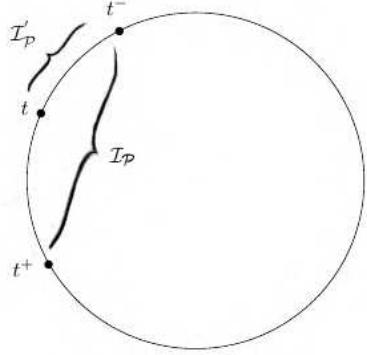}
\end{center}
\caption{The characteristic arc $\mathcal{I}_{\mathcal{P}^{\prime}}$ of the actual orbit portrait $\mathcal{P}^{\prime}$ is contained in the characteristic arc $\mathcal{I}_{\mathcal{P}}$ of the formal orbit portrait $\mathcal{P}$ and they share an endpoint, say $t^-$, of period $2p.$ }
\label{circle}
\end{figure}

Let, $\mathcal{A}_1^{\prime} = \lbrace t^-, t, t^+ \rbrace,$ where the characteristic angles are either $ \lbrace t^-, t \rbrace$ or  $\lbrace t , t^+ \rbrace$. Since $t(c) \in \mathcal{I}_{\mathcal{P}^{\prime}} \bigcap \mathcal{I}_{\mathcal{P}},$ we conclude that $t \in \left( t^-, t^+ \right)$. To fix our ideas, let's assume that $\mathcal{I}_{\mathcal{P}^{\prime}} =  \left( t^-, t \right)$. Since ${\mathcal{P}}^{\prime}$ is an orbit portrait satisfying the condition of Case (3) and $\left( t^-, t \right)$ is its characteristic arc, it follows from above that for any $c^{\prime} \notin \mathcal{M}_d^*$ with $t(c^{\prime}) \in \left( t^-, t \right)$, $f_{c^{\prime}}$ admits the orbit portrait ${\mathcal{P}}^{\prime}$.

Finally, to find an anti-polynomial admitting the orbit portrait ${\mathcal{P}}$, choose $\tilde{c} \notin \mathcal{M}_d^*$ with $t(\tilde{c}) = t$. Since $t(\tilde{c}) \in \left(t^-,t^+\right),$ the characteristic arc of the formal orbit portrait $\mathcal{P}$, it's now routine to check that the two dynamical rays 
$\mathcal{R}_{t^-}^{f_{\tilde{c}}}$ and $\mathcal{R}_{t^+}^{f_{\tilde{c}}}$ (of period $2p$) land at the same point $\tilde{z} \in J\left(f_{\tilde{c}}\right)$ such that $\tilde{z}$ has period $p$. Let $\tilde{\mathcal{P}} = \lbrace \tilde{\mathcal{A}_1} , \tilde{\mathcal{A}_2}, \cdots, \tilde{\mathcal{A}_{p}} \rbrace$ be the orbit portrait associated with the periodic orbit $\mathcal{O}(\tilde{z})$ with $t^- , t^+ \in \tilde{\mathcal{A}_1}$. We'll show that $\tilde{\mathcal{A}_1} = \lbrace t^- , t^+ \rbrace$, which will complete the proof.

First note that the dynamical ray $\mathcal{R}_{t}^{f_{\tilde{c}}}$ can't land because the assumption $t(\tilde{c}) = t$ forces the ray to bifurcate. If some other dynamical ray $\mathcal{R}_{\tilde{t}}^{f_{\tilde{c}}}$ ($\tilde{t} \neq t^- , t^+$) landed at $\tilde{z}$, then the angle $\tilde{t}$ must have period $p$ under multiplication by $-d$. Since $\tilde{z}$ is a repelling periodic point for $f_{\tilde{c}}$ and the dynamical rays at angles $\tilde{t}, t^- , t^+$ land at it, $\tilde{z}$ can be real-analytically followed as a repelling periodic point of odd period $p$ and the corresponding dynamical rays at the same angles would continue to land there under small perturbation of $\tilde{c}$. But if we choose $ c^{\prime} \approx \tilde{c}$ so that $t(c^{\prime}) \in  \left( t^-, t \right)$, then in the dynamical plane of $f_{c^{\prime}}$, the four rays at angles $\tilde{t}, t^- , t^+$ and $t$ would land at a common repelling periodic point of odd period: a contradiction to Lemma \ref{three}. We conclude that $\tilde{\mathcal{A}_1} = \lbrace t^- , t^+ \rbrace$ and hence $\tilde{\mathcal{P}} = \mathcal{P}$. 
\end{proof}

\begin{remark}
a) In case (4) of the previous theorem, we proved the realization of the orbit portrait $\mathcal{P}$ for parameters on a certain parameter ray. One can, with a bit more effort, make the following stronger statement: there is an open subset $S$ of $\mathbb{R}/\mathbb{Z}$ such that every parameter outside $\mathcal{M}_d^*$ having external angle in $S$ admits the orbit portrait $\mathcal{P}$. Indeed, if $\mathcal{I}_{\mathcal{P}}=\left(t^-,t^+\right)$ is the characteristic arc of the formal orbit portrait $\mathcal{P}$; then it is not hard to check that there are at most two angles $t_1, t_2 \in \left(t^-,t^+\right)$ of period $p$ (with $t_1<t_2$, say) such that the three  dynamical rays $\mathcal{R}_{t_i}^{f_{c}}$ ($i=1$ or $2$), $\mathcal{R}_{t^-}^{f_{c}}$ and $\mathcal{R}_{t^+}^{f_{c}}$ can possibly land at a common point. Also, this can happen precisely when the external angle $t(c)$ of $c$ lies in the intervals $(t^- , t_1)$ and $(t_2 , t^+)$ respectively. It follows that $S := (t_1 , t_2)$ satisfies the required property.

b) A unicritical anti-polynomial $f_c$ ($c \notin \mathcal{M}_d^*$) can admit an orbit portrait $\displaystyle $ $\mathcal{P} = \lbrace \mathcal{A}_1 , \mathcal{A}_2, \cdots$, $\mathcal{A}_{p} \rbrace$ only if $t(c) \in \left( t^- , t^+ \right)$, where $\left( t^- , t^+ \right)$ is the characteristic arc of $\mathcal{P}$. Indeed, the characteristic arc must be a critical value arc for some $\mathcal{A}_j$ and in the dynamical plane, the corresponding critical value sector bounded by the two rays $\mathcal{R}_{t^-}^{f_c}$ and $\mathcal{R}_{t^+}^{f_c}$ together with their common landing point contains the critical value $c$. Therefore the external angle $t(c)$ of $c$ will lie in the interval $\left( t^- , t^+ \right)$.
\end{remark}

\bibliographystyle{alpha}
\bibliography{antiorbitportraits}

\end{document}